\documentclass[11pt, a4paper]{amsart}
\usepackage{hyperref}
\usepackage[papersize={15cm,25cm},total={12.7cm,22.5cm},top=1.5cm,left=1.2cm]{geometry} 
\usepackage{amssymb, array, amsmath, amscd, subcaption, enumerate, amsthm, fixltx2e, setspace, mathtools}
\usepackage[utf8]{inputenc}
\usepackage{amssymb}
\usepackage{bm}
\usepackage{graphicx} 

\newcommand{\lapl}[1]
{\Delta #1}
\newcommand{\doo}[2]
{\frac{\partial #1}{\partial #2}}

\newtheorem{theorem}{Theorem}
\newtheorem{remark}[theorem]{Remark}

\begin{document}

\title[Robust Output Tracking for a Room Temperature Model]{Robust Output Tracking for a Room Temperature Model with Distributed Control and Observation}

\thispagestyle{plain}

\author{Konsta Huhtala, Lassi Paunonen and Weiwei Hu}
\address{(K. Huhtala and L. Paunonen) Mathematics and Statistic, Faculty of Information Technology and Communication Sciences, Tampere University, PO.\ Box 692, 33101 Tampere, Finland}
\email{konsta.huhtala@tuni.fi, lassi.paunonen@tuni.fi}
\address{(W. Hu) Department of Mathematics, University of Georgia, Athens, GA 30602, USA}
\email{Weiwei.Hu@uga.edu}

\begin{abstract}
	We consider robust output regulation of a partial differential equation model describing temperature evolution in a room. More precisely, we examine a two-dimensional room model with the velocity field and temperature evolution governed by the incompressible steady state Navier-Stokes and advection-diffusion equations, respectively, which coupled together form a simplification of the Boussinesq equations. We assume that the control and observation operators of our system are distributed, whereas the disturbance acts on a part of the boundary of the system. We solve the robust output regulation problem using a finite-dimensional low-order controller, which is constructed using model reduction on a finite element approximation of the model. Through numerical simulations, we compare performance of the reduced-order controller to that of the controller without model reduction as well as to performance of a low-gain robust controller.
\end{abstract}

\keywords{Linear control systems, robust control, output regulation, partial differential
	equations \newline The research was supported by the Academy of Finland Grant
	number 310489 held by L. Paunonen. L. Paunonen was funded by
	the Academy of Finland Grant number 298182. W. Hu was partially
	supported by the NSF grant DMS-1813570.} 

\maketitle

\section{INTRODUCTION}

We consider a thermal control problem for a two-dimensional room model. The partial differential equation (PDE) model considered consists of an advection--diffusion equation with the advection field governed by the steady state incompressible Navier--Stokes equations. This is a simplification, with one-way coupling between the PDEs, of the Boussinesq equations with two-way coupling between the advection field and the temperature.

Owing to the fact that these types of models are often employed to study thermal regulation and energy efficiency of buildings, modeling as well as theoretical approaches with both of the aforementioned setups are an active area of research. Feedback stabilization of both the Boussineq and the Navier--Stokes equations using boundary control has been studied by a variety of researchers, see for example \cite{Badra2012,Nguyen2015,Burns2016,Hu2016,He2018,Ramaswamy2019}, and several regulation examples for both PDE setups are presented in \cite{Aulisa2016}.

The goal of \emph{robust output tracking} is to have the output of the system converge to some desired reference trajectory despite disturbance signals or system perturbations. 
Achieving robust output tracking can be guaranteed by solving the \emph{robust output regulation problem} relying on the \emph{internal model principle}
\cite{Francis1975,Francis1976,Davison1976}. First developed in the 1970s for finite-dimensional systems, the principle has since been formulated also for infinite-dimensional systems, see \cite{Rebarber2003,Hamalainen2010,Paunonen2010} with recent results \cite{Paunonen2014,Humaloja2019} considering \emph{boundary control systems}.

As the main result of this paper, we construct a finite-dimensional low-order controller, which solves the robust output regulation problem for the room model and is based on the work of \cite{Paunonen2019}. Low order of the controller is achieved with two steps. As the first step, we approximate the plant using a finite element approximation and design the controller based on the received operator approximations. As the second step, we reduce the order of the controller by applying \emph{balanced truncation}, see \cite{Moore1981,Pernebo1982,Benner2013}.
As an alternative controller structure, and a point of reference, we construct a low-gain robust controller previously introduced in \cite{Hamalainen2000,Rebarber2003,Paunonen2016}.

The paper is organized as follows. In Section \ref{SPF}, we present the model describing robust thermal regulation problem of a two-dimensional room with distributed control and observation together with boundary disturbance. In Section 
\ref{STH}, we first formulate the robust output regulation problem and then verify certain properties of the PDE system.
The verified system properties are then utilized in Section \ref{Controldesign}, where we formulate the different controller designs used in this paper. Tracking performance of the controllers is illustrated in Section \ref{SNE} using numerical simulations before concluding the paper in Section \ref{SCO}.

We denote by $\mathcal{L}(X,Y)$ the set of bounded linear operators from a normed  space $X$ to a normed space $Y$. Expressions $\langle \cdot, \cdot \rangle_{\Omega}$ and $\langle \cdot, \cdot \rangle_{\Gamma}$ denote $L^2$-inner products on the two-dimensional domain $\Omega$ and one-dimensional domain $\Gamma$, respectively. 

\section{PROBLEM FORMULATION}\label{SPF}

\begin{figure}[h]
	\centering
	\includegraphics[scale=.3]{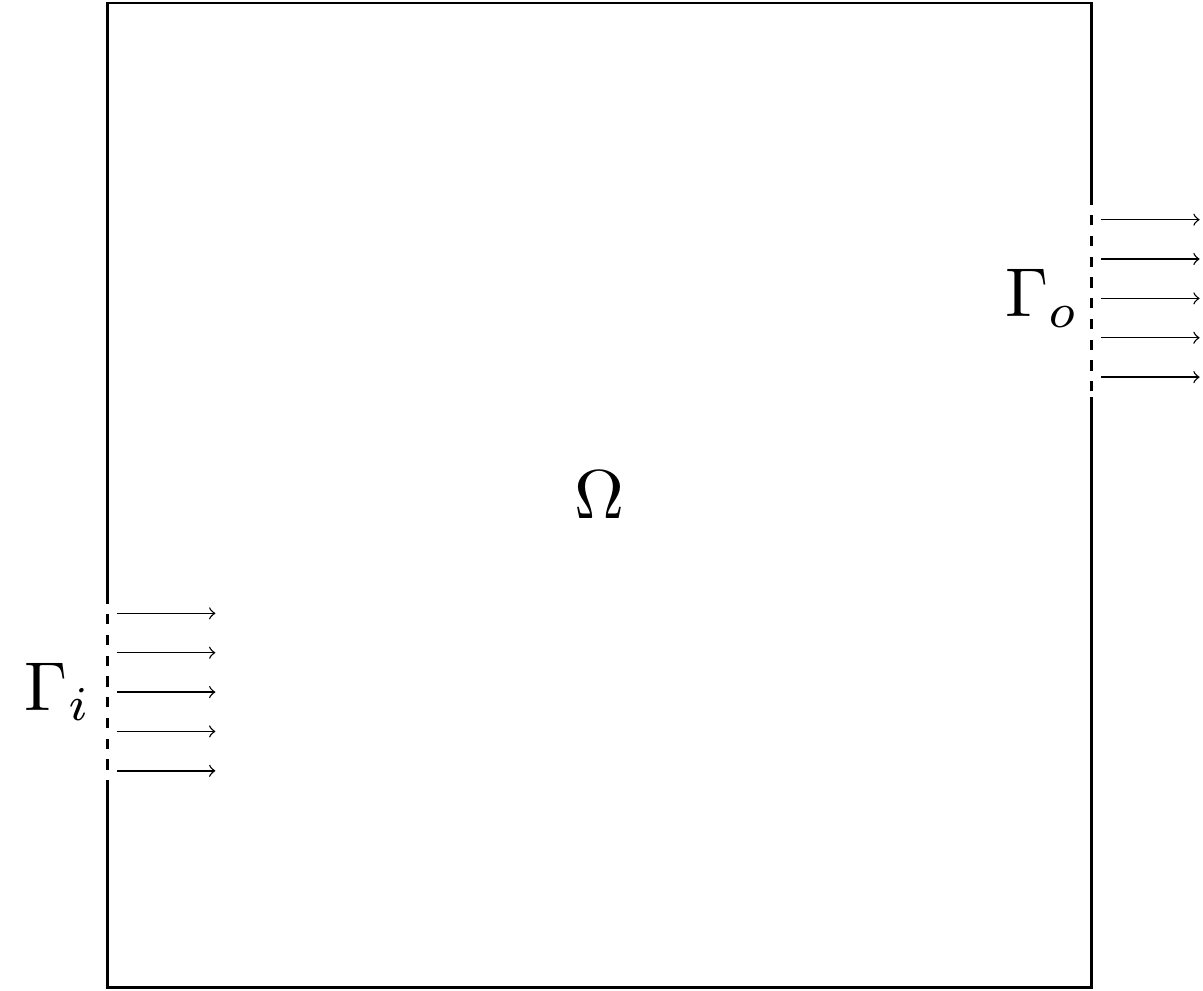}
	\caption{ Outline of the room}
	\label{Roompic}
\end{figure}
We consider a two-dimensional model of a rectangular room with an inlet and an outlet, see Fig. \ref{Roompic}. Denote the rectangle by $\Omega \subset \mathbb{R}^2$ and its piece-wise smooth boundary by $\Gamma$. Denote locations of the inlet and the oulet by $\Gamma_i \subset \Gamma$ and $\Gamma_o \subset \Gamma$, respectively, and assume $\Gamma_i \cap \Gamma_o = \emptyset$. Finally denote by $\Gamma_w = \Gamma \setminus (\Gamma_i \cup \Gamma_o)$ the walls of the room. We assume the temperature evolution of the room to be governed by the advection--diffusion equation coupled with the steady-state incompressible Navier--Stokes equations, i.e.
\begin{subequations}\label{System} 
	\begin{align}
	\dot{\theta}(\xi,t) &= \frac{1}{RePr}\lapl{\theta}(\xi,t)-v(\xi) \cdot \nabla \theta(\xi,t)+b(\xi)u(t),  \quad \theta(\xi,0) = \theta_0(\xi), \\
	0 &= \frac{1}{Re}\lapl{v}(\xi)-v(\xi)\cdot \nabla v(\xi) - \nabla p(\xi), \\
	0 &= \nabla \cdot v(\xi) \label{incompressible}
	\end{align} 
	subject to the boundary conditions
	\begin{align}
	\doo{\theta}{n}|_{\Gamma_i} &=b_dw_d, \qquad  \doo{\theta}{n}|_{\Gamma_o} = 0,
	\qquad \theta|_{\Gamma_w} = 0, \label{distbc} \\
	v|_{\Gamma_w} &= 0, \qquad v|_{\Gamma_i} = f, \qquad
	\big(\mathcal{T}(v,p)\cdot n\big)|_{\Gamma_o} = 0,
	\end{align}
\end{subequations}
where $\theta(\xi,t)$ is the fluid temperature, $v(\xi)$ is the fluid velocity, $p(\xi)$ is the fluid pressure, $Re$ is the Reynolds number, $Pr$ is the Prandtl number, $n$ is the unit outward normal vector of $\Gamma$, $w_d(t)$ is a disturbance signal 
applied according to the disturbance shape function $b_d(\xi)$, $f(\xi)$ is a time independent boundary fluid flux and $\mathcal{T}$ is the fluid Cauchy stress tensor. The chosen boundary conditions for $v$ are often called ``\emph{no-slip}'' for the $\Gamma_w$ part and ``\emph{stress-free}'' for the $\Gamma_o$ part. Finally, $u(t)$ is the input signal of the system
applied by $b(\xi)$.
We consider in-domain observations
\begin{align*}
y(t) = \int_{\Omega} \theta(\xi,t) c(\xi) d\xi,
\end{align*}
where $c(\xi)$ is a weight function. Note that depending on $Re$ and $Pr$, the choices for $b(\xi)$ and $c(\xi)$ may be restricted to guarantee exponential stabilizability and exponential detectability of the system, c.f. Section \ref{Controldesign}.

The reference signals $y_r(t)$ to be tracked and the disturbance signals $w_d(t)$ to be rejected are of the forms
\begin{subequations}\label{Signals}
	\begin{align}
	y_r(t) = \sum_{n=1}^{q}(a_n^c\cos(\omega_nt)+a_n^s\sin(\omega_nt)), \\
	w_d(t) = \sum_{n=1}^{q}(b_n^c\cos(\omega_nt)+b_n^s\sin(\omega_nt)),\label{distsig}
	\end{align}
\end{subequations}
where $\omega_n$ are known frequencies and $a_n^c,a_n^s \in \mathbb{R}^{p_s}$, $b_n^c$, $b_n^s \in \mathbb{R}^d$ are possibly unknown coefficients.
Our goal, more precisely defined in Section \ref{AFS}, is to have $y(t)$ converge exponentially to a given reference $y_r(t)$ despite the disturbance $w_d(t)$. 
\begin{remark}\label{rem:signals}
	The controller designs of this paper can be used for a larger class of signals than presented in \eqref{Signals}, as well as for setups with more inputs than outputs. To be precise, signals with time-dependent polynomial coefficients $a_k^i(t),b_k^j(t)$, $k=1,...,q$ can be handled, see \cite{Hamalainen2000,Paunonen2016}.
\end{remark}

\section{ABSTRACT FORMULATION OF THE CONTROL PROBLEM}\label{STH}

We first present the robust output regulation problem for \emph{abstract linear systems} with bounded control, observation and disturbance. Then we proceed to show that the considered room model fits into this framework.

\subsection{Abstract Linear Control Systems}\label{AFS}

Let $X$ be a Hilbert space. We formulate the plant of a control system as an abstract linear system
\begin{subequations}\label{ACP}
	\begin{align}
	\hspace{-1em}\dot{x}(t) &= Ax(t) + Bu(t) + B_dw_d(t), \quad x(0) = x_0 \in X, \\
	\hspace{-1em} y(t) &= Cx(t) + Du(t) + D_dw_d(t),
	\end{align}
\end{subequations}
where $x(t)$ and $x_0$ are the state and the initial state of the system, $A:D(A) \subset X \longrightarrow X$ is the generator of a strongly continuous semigroup, $B \in \mathcal{L}(U,X)$ is the control operator, $B_d \in \mathcal{L}(U_d,X)$ is the disturbance operator, $C \in \mathcal{L}(X,Y)$ is the observation operator and $D \in \mathcal{L}(U,Y)$ and $D_d \in \mathcal{L}(U_d,Y)$ are the feedthrough operators for the input and the disturbance signals, respectively. Here $U=\mathbb{R}^m$, $Y= \mathbb{R}^{p_s}$ and $U_d = \mathbb{R}^d$. Note that we assume for the system operators other than $A$ to be bounded. 

The \emph{dynamic error feedback controller} on a Hilbert space $Z$ is formulated as 
\begin{subequations}\label{Controller}
	\begin{align}
	\dot{z}(t) &= \mathcal{G}_1z(t) + \mathcal{G}_2e(t), \qquad z(0) = z_0 \in Z, \\
	u(t) &= Kz(t),
	\end{align}
\end{subequations}
where $z(t)$ and $z_0$ are the state and the initial state of the controller, $\mathcal{G}_1:D(\mathcal{G}_1)\subset Z \longrightarrow Z$ is the generator of a strongly continuous semigroup on $Z$, $\mathcal{G}_2 \in \mathcal{L}(Y,Z)$, $e(t)=y(t)-y_r(t)$ is the regulation error and $K \in \mathcal{L}(Z,U)$. For practicality, we aim to have $Z$ be finite-, preferably low-dimensional for efficient applicability of the controller.

Finally, we define the \emph{closed-loop system}, see \cite{Paunonen2010}, consisting of the plant and the controller with the state $x_e = (x(t),  z(t))^T$ and the initial state $x_{e0} = (x_0,z_0)^T \in X \times Z \eqqcolon X_e$ as
\begin{align*}
\dot{x}_e(t) &= A_ex_e(t)+B_ew_e(t), \qquad x_e(0) = x_{e0}, \\
e(t) &= C_ex_e(t) + D_ew_e(t),
\end{align*}
where $w_e(t) = \begin{bmatrix}
w_d(t), & y_r(t)
\end{bmatrix}^T$,
\begin{align*}
A_e &= \begin{bmatrix}
A & BK \\ \mathcal{G}_2C & \mathcal{G}_1+\mathcal{G}_2DK
\end{bmatrix},
\qquad
B_e = \begin{bmatrix}
B_d & 0 \\ \mathcal{G}_2D_d & -\mathcal{G}_2
\end{bmatrix}, \\
C_e &= \begin{bmatrix}
C, & DK
\end{bmatrix},
\qquad
D_e = \begin{bmatrix}
D_d, & -I
\end{bmatrix}
\end{align*}
and $A_e$ generates a strongly continuous semigroup $T_e(t)$ on $X_e$.

The robust output regulation problem for systems of the form \eqref{ACP} can be formulated as follows.
\vskip .1cm
\noindent\textbf{The Robust Output Regulation Problem}. Design a controller of the form \eqref{Controller} such that
\begin{enumerate}[(i)]
	\item The semigroup $T_e(t)$ is exponentially stable.
	\item For every $x_{e0}$ and $y_r(t),w_d(t)$ of the form \eqref{Signals},
	\begin{align}\label{tracking}
	||e(t)|| \le M_ee^{-\omega_et}(||x_{e0}||+||\Lambda||)
	\end{align}
	for some $M_e,\omega_e > 0$, where vector $\Lambda$ contains the constants $\{a_n^c\}_n$, $\{a_n^s\}_n$, $\{b_n^c\}_n$, $\{b_n^s\}_n$.	
	\item If the operators $A,B,B_d,C,D,D_d$ are perturbed in a way such that the perturbed closed-loop system remains exponentially stable, then \eqref{tracking} is still satisfied for all $x_{e0}$ and $y_r(t),w_d(t)$ of the form \eqref{Signals} for some $M_{ep},\omega_{ep}>0$. 
\end{enumerate}
By the internal model principle, see \cite{Paunonen2010}, a controller with $\mathcal{G}_1$ including an internal model based on the frequencies $\{\omega_n\}_{n=1}^q$ and $\mathcal{G}_2,K$ chosen such that the closed-loop system is exponentially stable solves the robust output regulation problem.

\subsection{The Room Model as a Control System}\label{ContSys}
To represent the room model \eqref{System} as an abstract control system,
we assume $b,c \in L^2(\Omega)$ and $b_d \in L^2(\Gamma_i)$, define
\begin{align*}
X &= L^2(\Omega),  \qquad 
H_\theta = \big\{\theta \in H^1(\Omega)\big|\ \theta|_{\Gamma_w} = 0 \big\},
\\B &= b(\xi) \in \mathcal{L}(U,X), \qquad C = \langle \cdot, c \rangle_{\Omega} \in \mathcal{L}(X,Y)
\end{align*}
and note that a steady state solution $(v_e,p_e) \in \big\{v \in (H^1(\Omega))^2 \big|\ \nabla \cdot v = 0, \ v|_{\Gamma_w} = 0 \big\} \times L^2(\Omega)$ for the Navier-Stokes equations in \eqref{System} is guaranteed to exist by \cite{Mazya2009}. Define the bilinear form
\begin{equation*}
a(\theta,\phi) = \alpha \langle \nabla \theta,\nabla \phi \rangle_{\Omega} + \langle v_e \cdot \nabla \theta, \phi \rangle_{\Omega} \qquad \forall \theta,\phi \in H_\theta,
\end{equation*}
where $\alpha \coloneqq 1/(RePr)$, and define the operator $A$ by
\begin{align*}
&\langle A\theta,\phi \rangle_{\Omega} = -a(\theta,\phi), \\
&D(A) = \big\{ \theta \in H_\theta \big| \forall \phi \in H_\theta, \ \phi \to a(\theta,\phi) \textrm{ is $L^2$-continuous} \big\}.
\end{align*}
\begin{theorem}
	The bilinear form $a(\cdot, \cdot)$ is $H_\theta$-bounded and $H_\theta$-coercive, and the operator $A$ generates an analytic semigroup on $X$.
\end{theorem}
\begin{proof}
	Let  each $k_{(\cdot)}$ denote a constant.
	For $\theta,\phi \in H_\theta$, using integration by parts on the advection term and recalling $\nabla \cdot v_e = 0$, we get
	\begin{align*}
	|a(\theta,\phi)| 
	&\le \alpha|\langle \nabla \theta, \nabla \phi \rangle_{\Omega}| + |\langle v_e \cdot \nabla \theta,\phi \rangle_{\Omega}|
	\\&\le \alpha|\langle \nabla \theta, \nabla \phi \rangle_{\Omega}| + |\langle v_e \cdot n , \theta \phi \rangle_{\Gamma}|+ |\langle v_e \theta, \nabla \phi \rangle_{\Omega}|.
	\end{align*}	
	By Sobolev embedding theorems, for $\phi,\psi \in H^1(\Omega)$
	\begin{equation}\label{Sembedding}
	||\phi\psi||_{L^2(\Omega)}
	\le k_1||\phi||_{H^1(\Omega)}||\psi||_{H^1(\Omega)},
	\end{equation}
	which together with properties of the trace operator and $L^2$-duality of $H^{\frac{1}{2}}$ and $H^{-\frac{1}{2}}$ implies
	\begin{equation*}
	|\langle v_e \cdot n,\theta \phi \rangle_{\Gamma}| 
	\le k_2 ||v_e||_{H^1(\Omega)}||\theta||_{H^1(\Omega)}||\phi||_{H^1(\Omega)}.
	\end{equation*}
	Applying \eqref{Sembedding} for the third term and using Poincare's inequality, we finally get
	\begin{equation*}
	|a(\theta,\phi)| 
	\le k_3 ||v_e||_{H^1(\Omega)}||\theta||_{H^1(\Omega)}||\phi||_{H^1(\Omega)},
	\end{equation*}
	thus $a(\cdot,\cdot)$ is $H_\theta$-bounded.
	
	By Ladyzhenskaya's and Young's inequalities,
	\begin{align*}
	|\langle v_e \cdot n,\theta^2 \rangle_{\Gamma}| 
	&\le k_4||v_e||_{H^1(\Omega)}||\theta||_{L^2(\Omega)}||\nabla\theta||_{L^2(\Omega)}
	\\&\le \frac{k_4^2}{4\alpha} ||v_e||_{H^1(\Omega)}^2||\theta||_{L^2(\Omega)}^2 + \alpha ||\nabla\theta||_{L^2(\Omega)}^2,
	\end{align*}
	therefore
	\begin{align*}
	a(\theta,\theta) 
	&= \alpha ||\nabla \theta||^2_{L^2(\Omega)} + \frac{1}{2}\langle v_e \cdot n,\theta^2 \rangle_{\Gamma}
	\\&\ge \alpha ||\nabla \theta||^2_{L^2(\Omega)} - \frac{k_4^2}{8\alpha}||v_e||^2_{H^1(\Omega)}||\theta||^2_{L^2(\Omega)} - \frac{\alpha}{2}||\nabla\theta||^2_{L^2(\Omega)}
	\end{align*}
	and hence, for $\lambda \ge \frac{k_4^2}{8}||v_e||^2_{H^1(\Omega)}$
	\begin{equation*}
	a(\theta,\theta) + \lambda||\theta||^2_{L^2(\Omega)} \ge \frac{\alpha}{2} ||\theta||_{H^1(\Omega)}^2,
	\end{equation*}	
	which indicates that $a(\cdot,\cdot)$ is $H_\theta$-coercive. This in turn implies that $A$ generates an analytic semigroup on $X$, see \cite{Banks1997}.
\end{proof}

Due to the disturbance signal $w_d$ being applied on the boundary via $b_d$, the corresponding disturbance operator $\tilde{B}_d$ is not bounded from $U_d$ to $X$. However, since the disturbance signals \eqref{distsig} are smooth, we can use a change of variable $x = \tilde{x} - B_{di}w_d$ to homogenize the disturbance boundary condition \eqref{distbc}, c.f. \cite[Ch. 3.3]{Curtain1995}. Here we have denoted the original state variable by $\tilde{x}$, the final state variable by $x$ and $B_{di}$ is a right inverse of $\doo{(\cdot)}{n} |_{\Gamma_i}$. The change of variable leads to a bounded disturbance operator $B_d$ while also introducing a bounded disturbance feedthrough operator $D_d$.
As such, after the change of variable the room model can be presented in the form \eqref{ACP} with the operators other than $A$ being bounded.
As we will see in Section \ref{Controldesign}, the operators $B_d$ and $D_d$ are not used for the controller construction. Thus it is enough to know that the room model can be expressed as \eqref{ACP} with bounded operators $B_d$ and $D_d$, but we do not need an exact expression for these operators.

\section{CONTROLLER DESIGN}\label{Controldesign}
We present two different controller designs to solve the robust output regulation problem for systems of the form \eqref{ACP}, which we now know includes the room model \eqref{System}. The first one is the ``main'' controller design of the paper, while the second one is constructed mainly for performance comparison purposes.

The ``main'' controller structure, called ``dual observer-based finite-\linebreak[4] dimensional controller'', makes use of a Galerkin approximation to design an observer-based yet finite-dimensional controller.
To further decrease the size of the controller, balanced truncation is included as a part of the controller design process.
From now on, $(\cdot)^N$ refers to the Galerkin approximation of an operator and $(\cdot)^r$ refers to an operator obtained through model reduction by balanced truncation.

For the details on the controller design, as well as a short introduction to the Galerkin approximation and balanced truncation, see \cite{Paunonen2019}. For a more complete theory, see for example \cite{Strang2008,Benner2013}. When designing  the controllers, we make the following assumptions.
\begin{itemize}
	\item The pair $(A,B)$ is exponentially stabilizable and $(A,C)$ is exponentially detectable.
	\item For a finite-dimensional approximating subspace $H_\theta^N$ of $H_\theta$,
		\begin{align}\label{Galerkin}
		&\forall \phi \in H_\theta \ \exists (\phi^N)_N, \ \phi^N \in H_\theta^N: 
		||\phi^N-\phi||_V \xrightarrow{N \to \infty} 0.
		\end{align}
	\item Orders $N$ and $r \le N$ are chosen large enough.
	\item For simplicity and recalling Remark \ref{rem:signals}, the system has equal number of inputs and outputs.
\end{itemize}
Note that since the operator $A$ has only point spectrum and a finite number of positive eigenvalues each with finite multiplicity, the stabilizability and detectability considerations can be reduced to controllability and observability checks of the finite-dimensional unstable parts, see \cite{Badra2014,Burns2016}.
\vskip .1cm
\noindent\textbf{The Dual Observer-Based Finite-Dimensional Controller}. Define a controller of the form \eqref{Controller} with $z(t)=\begin{bmatrix}z_1(t), & z_2(t)
\end{bmatrix}^T$ by 
\begin{subequations}\label{dcontroller}
	\begin{align}
	\dot{z}_1(t) &= G_1z_1(t) + G_2^NC_K^rz_2(t) + G_2^Ne(t), \\
	\dot{z}_2(t) &= (A_K^r + L^rC^r_K)z_2(t) + L^re(t), \\
	u(t) &= K_1z_1(t) - K_2^rz_2(t),
	\end{align}
\end{subequations}
where the operators are chosen as follows, see \cite{Paunonen2019}.
Choose positive matrices $R_1 \in \mathcal{L}(U)$, $R_2 \in \mathcal{L}(Y)$, constants $\alpha_1,\alpha_2 > 0$ and operators $Q_1 \in \mathcal{L}(X,Y_0)$, $Q_2 \in \mathcal{L}(U_0,X)$ with Hilbert spaces $U_0,Y_0$, such that $(A+\alpha_1I,B,Q_1)$ and $(A+\alpha_2,Q_2,C)$ are both exponentially stabilizable and detectable.
Define
\begin{align*}
G_1 &= \textrm{diag}(J_1^Y,J_2^Y,...,J_q^Y) \ \mbox{with} \ J_k^Y = \begin{bmatrix}
0_p & \omega_kI_p \\
-\omega_kI_p & 0_p
\end{bmatrix}, \\
K_1 &= (K_1^k)_{k=1}^q, \qquad K_1^k = [I_p, \ 0_p], \\
\mathcal{G}_2^N &=
\begin{bmatrix}
G^N_2 \\ L^N
\end{bmatrix} = -\Pi_NC_s^NR_2^{-1}, \quad
K_2^N = -R_1^{-1}\big(B^N\big)^*\Sigma_N, \\
A_s^N &= \begin{bmatrix}
G_1 & 0 \\ B^NK_1 & A^N
\end{bmatrix}.
\end{align*}
Here $\Sigma_N,\Pi_N$ are the non-negative solutions of the finite-dimensional Riccati equations
\begin{align*}
&(A^N+\alpha_1I)^*\Sigma_N + \Sigma_N(A^N+\alpha_1I)-\Sigma_NB^NR_1^{-1}(B^N)^*\Sigma_N=-(Q_1^N)^*Q_1^N, \\
&(A_s^N+\alpha_2I)\Pi_N+\Pi_N(A_s^N+\alpha_2)^*-\Pi_N(C_s^N)^*R_2^{-1}C_s^N\Pi_N = -Q_2^N(Q_2^N)^*.
\end{align*}
Finally, use balanced truncation on the system
\begin{align*}
\Big(A^N+B^NK_2^N,L^N,\begin{bmatrix}
C^N+DK_2^N \\ K_2^N
\end{bmatrix}\Big)
\end{align*}
to obtain a stable $r$-dimensional system
\begin{align*}
\Big( A_K^r,L^r,\begin{bmatrix}
C_K^r \\ K_2^r
\end{bmatrix} \Big).
\end{align*}
With the aforementioned choices, \eqref{dcontroller} solves the robust output regulation problem for signals of the form \eqref{Signals} provided that $N,r\le N$ are large enough.

The second controller design is considerably more straightforward to construct. However, it requires the plant to be exponentially stable instead of only exponentially stabilizable. Note that this controller design can also be used for initially unstable systems as long as we manage to first stabilize them and only afterwards design the controller.

\vskip .1cm
\noindent\textbf{The Low-Gain Robust Controller}.
Assuming that $A$ generates an exponentially stable semigroup on $X$, choose the operators in \eqref{Controller} as
\begin{subequations}\label{econtroller}
	\begin{align}
	\mathcal{G}_1 &= G_1, \qquad \mathcal{G}_2 = (\mathcal{G}_2^k)_{k=1}^{q}, \qquad \mathcal{G}_2^k = \begin{bmatrix}
	-I_p & 0_p
	\end{bmatrix}^T, \\
	K &= \epsilon K_0 = \epsilon(K_0^1,...,K_0^q),\\
	K_0^k &= \begin{bmatrix}
	\textrm{Re}(P(i\omega_k)^{-1}) & \textrm{Im}(P(i\omega_k)^{-1})
	\end{bmatrix}, 
	\end{align}
\end{subequations}
where $P(\cdot)$ is the \emph{transfer function} of the plant \eqref{ACP}
and $\epsilon>0$ is a parameter used to tune the controller. Note that besides stability, the only information of the plant required for this controller is that of the transfer function values at certain frequencies, hence it is easy to implement.
As is shown in \cite{Paunonen2016}, this simple controller structure solves the robust output regulation problem for signals of the form \eqref{Signals} and small enough $\epsilon>0$.

\section{NUMERICAL EXAMPLE}\label{SNE}

In this section we illustrate the two controller structures in action. We consider a case with the control and disturbance profiles given by
\begin{align*}
b(\xi) = \chi_{[0,0.05] \times [0.1,0.4]}(\xi), \qquad b_d(\xi) = \chi_\Gamma|_{\Gamma_i}(\xi).
\end{align*}
Existence of $B_{di}$ can be verified according to \cite[Ch. 10]{TucsnakWeiss}.
Observation weights are given by one of the two options
\begin{subequations}
	\begin{align}
	c^{(1)}(\xi) &= 0.2^{-2}\chi_{[0.7,0.9] \times [0.1,0.3]}(\xi),\label{setup1} \\
	c^{(2)}(\xi) &= 0.2^{-2}\chi_{[0.1,0.3] \times [0.7,0.9]}(\xi),\label{setup2}
	\end{align}
\end{subequations}
which yield us two different setups $(b,c^{(1)})$ and $(b,c^{(2)})$.

Let the room be defined by $\Omega = [0,1] \times [0,1]$, $\Gamma_i = \{\xi_1 = 0, \ 0.1 \le \xi_2 \le 0.4\}$ and $\Gamma_o = \{ \xi_1 = 1, \ 0.5 \le \xi_2 \le 0.9 \}$. We choose $Re = 100$, $Pr=0.7$ and consider the inlet velocity profile
\begin{align*}
f(\xi) = \begin{bmatrix}
\exp(-\frac{0.0001}{((0.5-\xi_2)(0.9-\xi_2))^2}, & 0
\end{bmatrix}^T.
\end{align*}
With these choices, we check numerically that the analytic semigroup generated by $A$ is exponentially stable, thus no stabilization step is required for the low-gain controller.

As the first step, we need to calculate the steady state solution for the incompressible Navier--Stokes advection field. For this we employ the Taylor-Hood finite element scheme. That is, we triangulate $\Omega$ and use quadratic shape functions for the velocity approximation and linear shape functions for the pressure approximation. We calculate the steady state solution using Newton's method with the initial guess given by the steady state incompressible Stokes equation. Quadratic shape functions, which satisfy \eqref{Galerkin} by \cite{Strang2008}, are then also used for temperature approximation.
The mesh for quadratic shape functions contains $81$ nodes in each direction with uniform spacing. The calculated steady state advection field is depicted in Fig. \ref{vfield} with the ends of the inlet highlighted with black dots and the ends of the outlet highlighted with red dots.
\begin{figure}[h]
	\begin{subfigure}[]{0.44\columnwidth}
		\includegraphics[width=\linewidth,height=5cm]{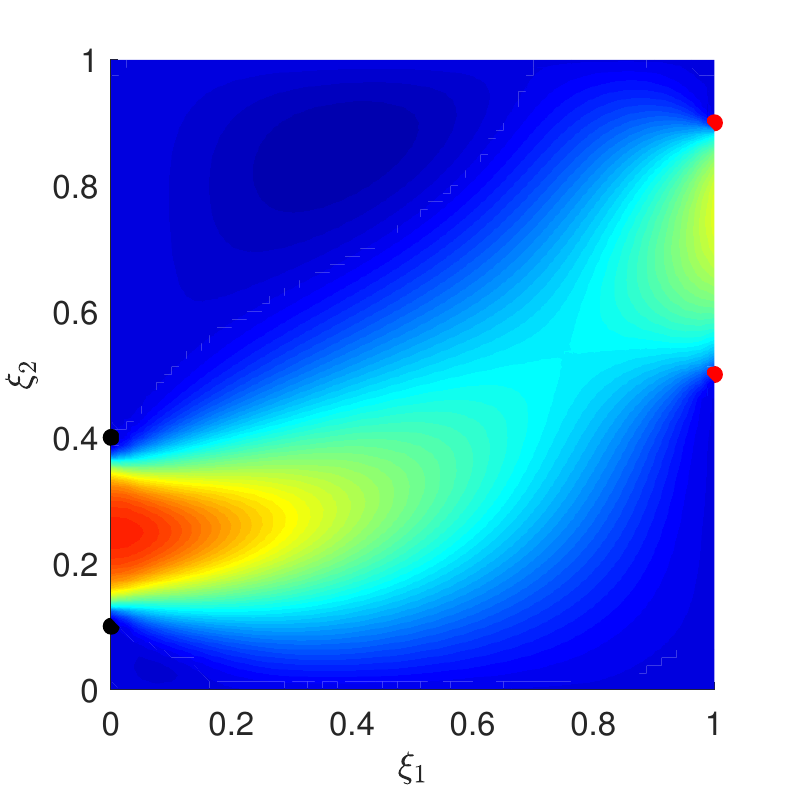}
	\end{subfigure} 
	\begin{subfigure}[]{0.55\columnwidth}
		\includegraphics[width=\linewidth,height=5cm]{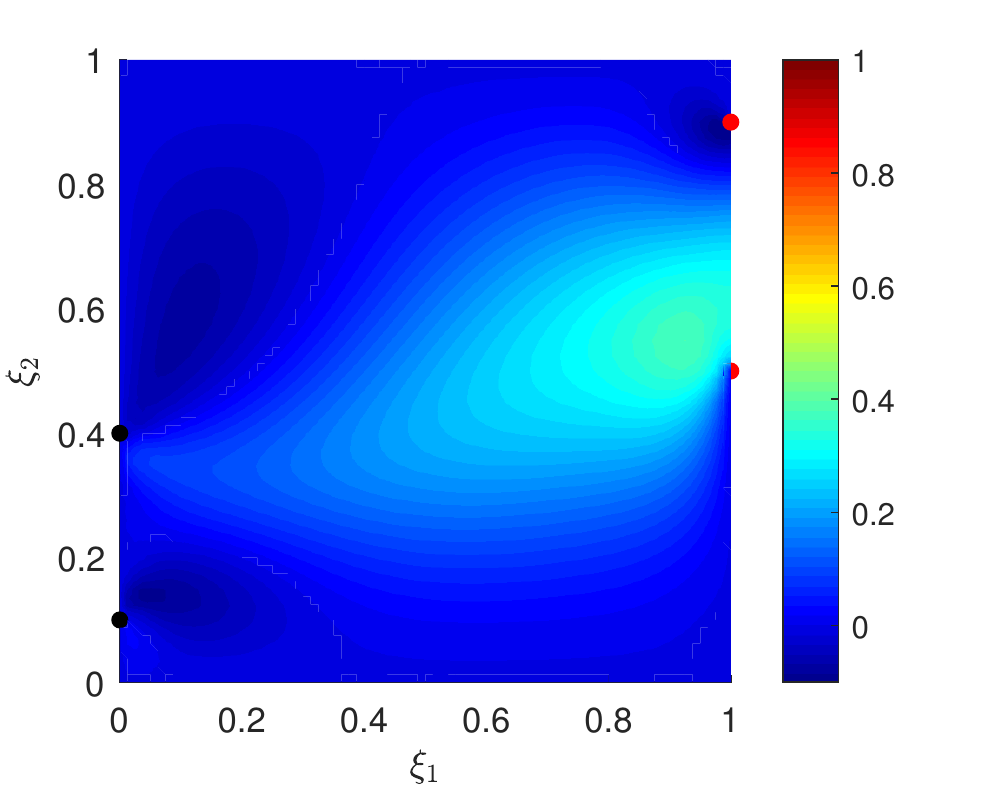}
	\end{subfigure} \hfill
	\caption{Steady state velocity field, $v_{e\xi_1}$ on the left and $v_{e\xi_2}$ on the right}
	\label{vfield}
\end{figure}

We choose as the reference and disturbance signals
\begin{align*}
y_r(t) = \sin(t) + 2\cos(2t), \qquad w_d(t) = 1.5\cos(3t), 
\end{align*}
which leads to an internal model of dimension $Z_0=6$ in our controller designs. Using the full mesh to approximate the infinite-dimensional plant leads, after accounting for the boundary conditions, to a model order of $6297$.

For the dual observer-based controller, we then use a sparser mesh with 41 nodes in each direction for quadratic shape functions to construct the operators $(\cdot)^N$ yielding a plant approximation of order $1549$, and reduce the order down to $r=10$ using the \texttt{balred} MATLAB function to construct the operators $(\cdot)^r$. Thus the controllers are of sizes $\dim Z_f = 1555$, $\dim Z_r = 16$ and $\dim Z_{lg} = 6$ for the full, reduced and low-gain controller, respectively. 
For each observation setup, we test two parameter sets for the controller \eqref{dcontroller};
\begin{subequations}
	\begin{align}
	\alpha_1&=\alpha_2=R_1=R_2=1,\label{choice1} \\
	\alpha_1'&=\alpha_2'=0.5, \qquad R_1'=R_2'=100, \label{choice2}
	\end{align}
\end{subequations}
with $Q_1=Q_2=I$ for both parameter sets.

For the low-gain robust controller, we  use the full model of order $6297$ to approximate the transfer function
\begin{align*}
P(s) = C(sI-A)^{-1}B,
\end{align*}
and choose $\epsilon_1=0.08$, $\epsilon_2=0.05$ for the observation setups \eqref{setup1} and \eqref{setup2}, respectively. These choices are based on roughly maximizing the stability margin of the closed-loop system.
Finally, we choose $x_{e0}=\begin{bmatrix}
1_X, & 0_Z
\end{bmatrix}^T$ as the initial state for each of the simulations.

The tracking performances of the controllers, illustrated in terms of the system outputs for the setup \eqref{setup1} and in terms of the tracking errors for the setup \eqref{setup2}, are presented in Figs \ref{error1} and \ref{error2}. The figures include either  the observations $y_{(\cdot)}(t)$ (Fig. \ref{error1}) or the errors $e_{(\cdot)}(t)$ (Fig. \ref{error2}) for five different controllers: $y_{red}(t),e_{red}(t)$ and $y_{red}'(t),e_{red}'(t)$ for \eqref{dcontroller} with model reduction, $y_f(t),e_f(t)$ and $y_f'(t),e_f'(t)$ for \eqref{dcontroller} without model reduction and $y_{lg}(t),e_{lg}(t)$ for the low-gain controller.

We observe for the performance difference between the model reduced and the full versions of \eqref{dcontroller} to be negligible at least for the chosen closed-loop initial state and $r$, since $\sup_{t\in [0,20]}|y_f(t)-y_{red}(t)| \lessapprox 0.01$ across all of the observation and parameter variants.
For the observation setup \eqref{setup1}, there are no major differences in the tracking performance between the parameter choices, but for \eqref{setup2} the choice \eqref{choice1} converges faster. 
For both setups, error convergence of the low-gain controller is the slowest despite the attempted optimization of $\epsilon_i$.
\begin{figure}[h]
	\begin{center}
		\includegraphics[width=8.63cm,height=3.5cm]{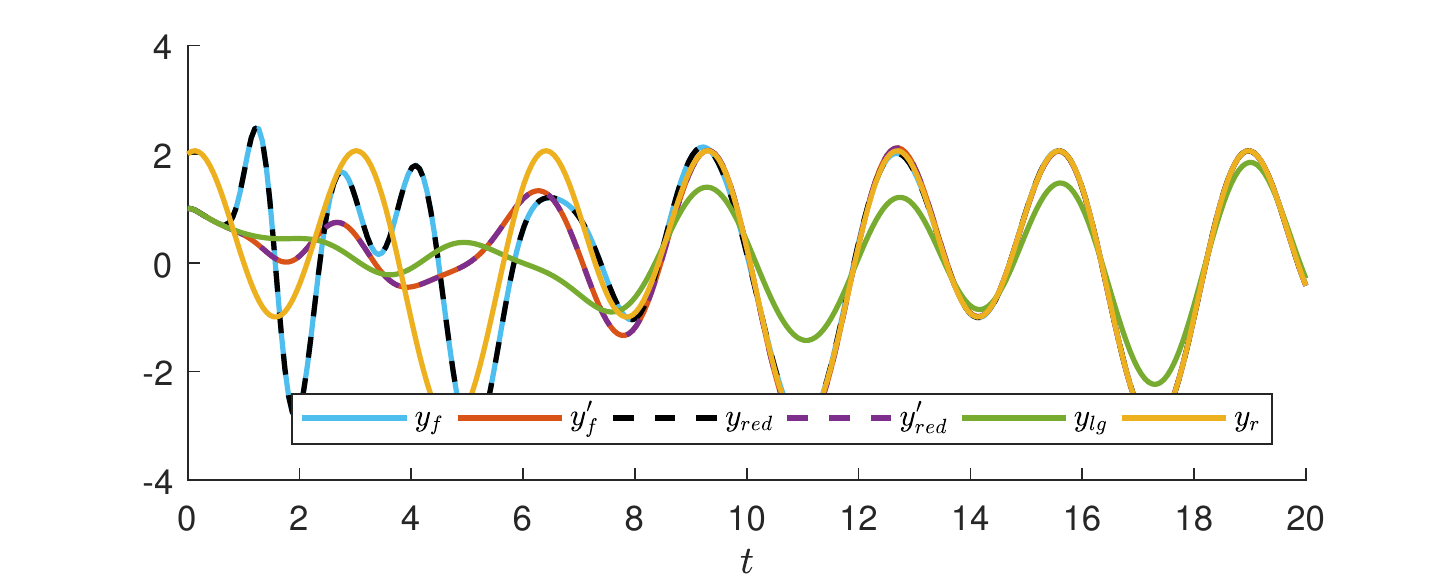}
		\caption{Plant outputs with the observation setup \eqref{setup1} for the different controllers} 
		\label{error1}
	\end{center}
\end{figure}
\begin{figure}[h]
	\begin{center}
		\includegraphics[width=8.63cm,height=3.5cm]{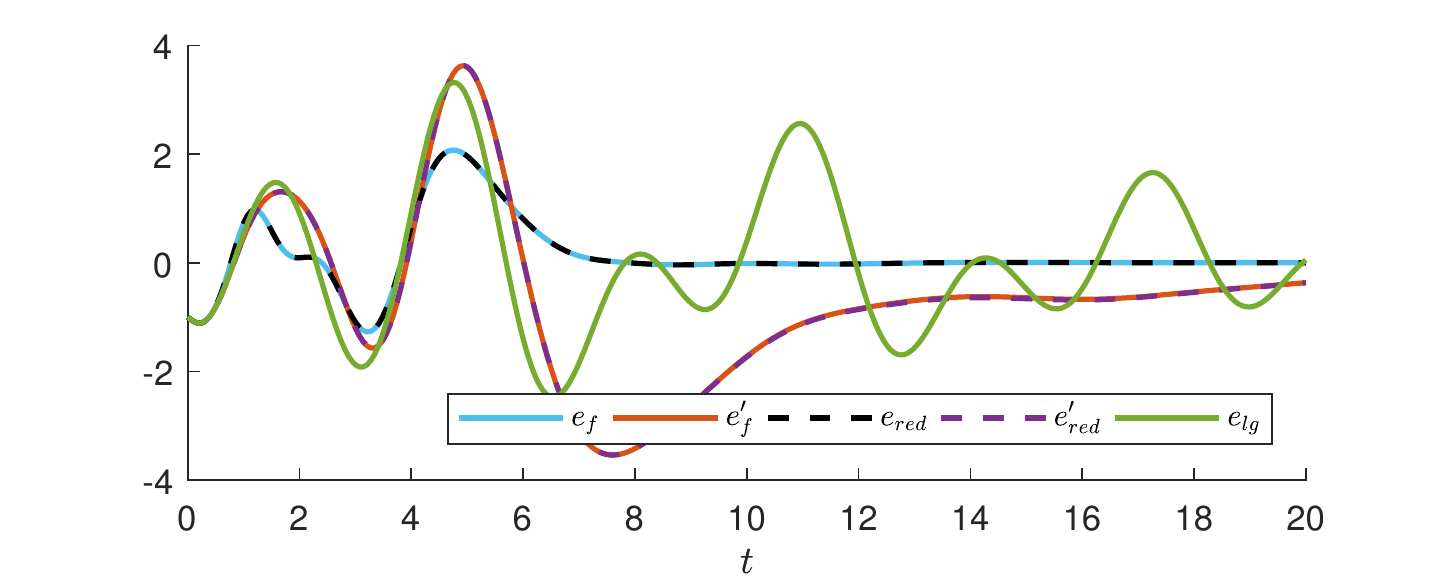}
		\caption{Tracking errors with the observation setup \eqref{setup2} for the different controllers} 
		\label{error2}
	\end{center}
\end{figure}

The plant states at time $t=20$ for both of the setups using controller \eqref{dcontroller} with model reduction and the choice \eqref{choice1} are presented in Fig. \ref{state}. For both of the state plots, the ends of the outlet are highlighted with red dots, the corners of the observed area by magenta dots and the corners of the controlled area, coinciding with the ends of the inlet at the boundary, with black dots.
\begin{figure}[h]
	\begin{subfigure}[]{0.49\columnwidth}
		\includegraphics[width=\linewidth,height=4.5cm]{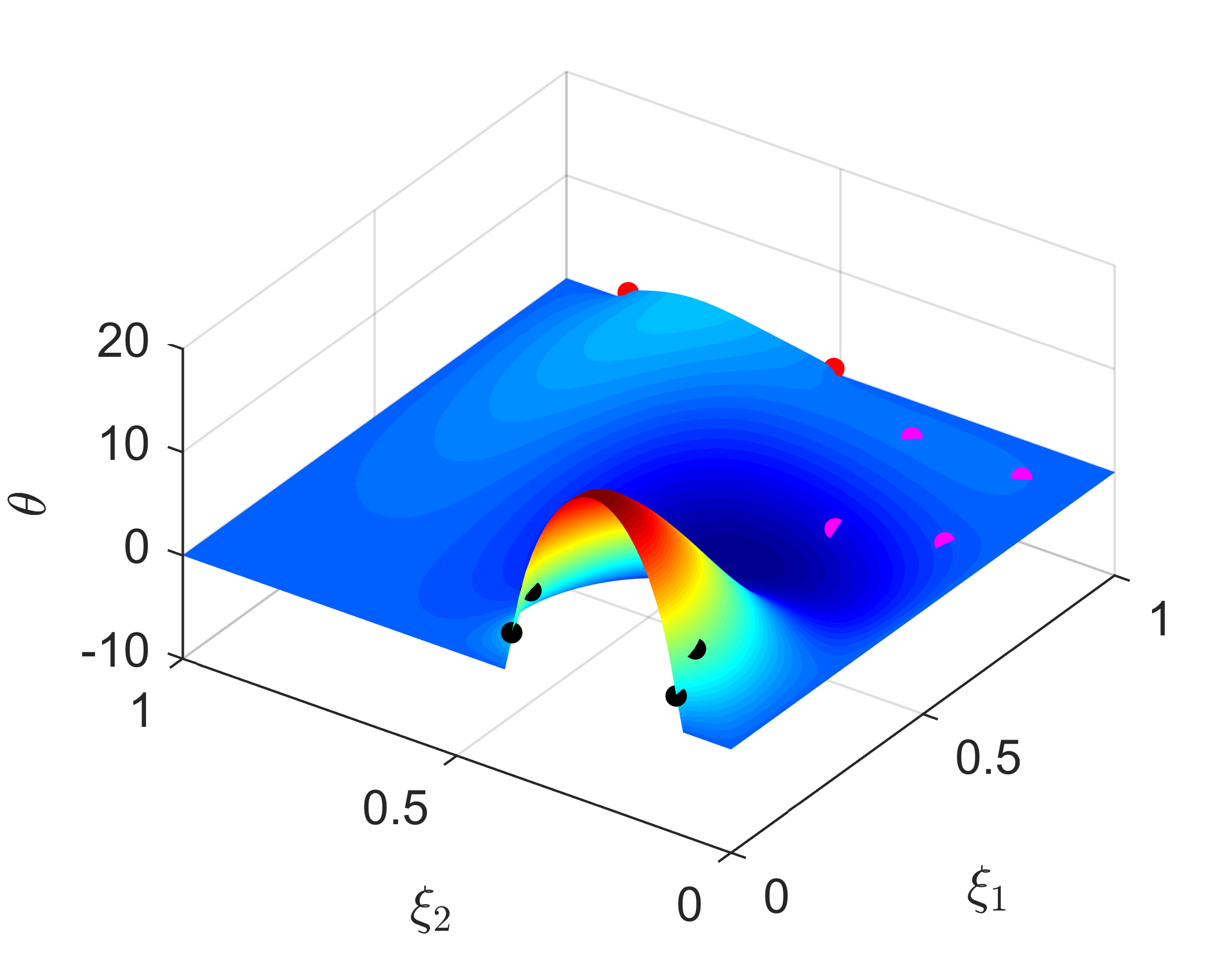}
	\end{subfigure} 
	\begin{subfigure}[]{0.49\columnwidth}
		\includegraphics[width=\linewidth,height=4.5cm]{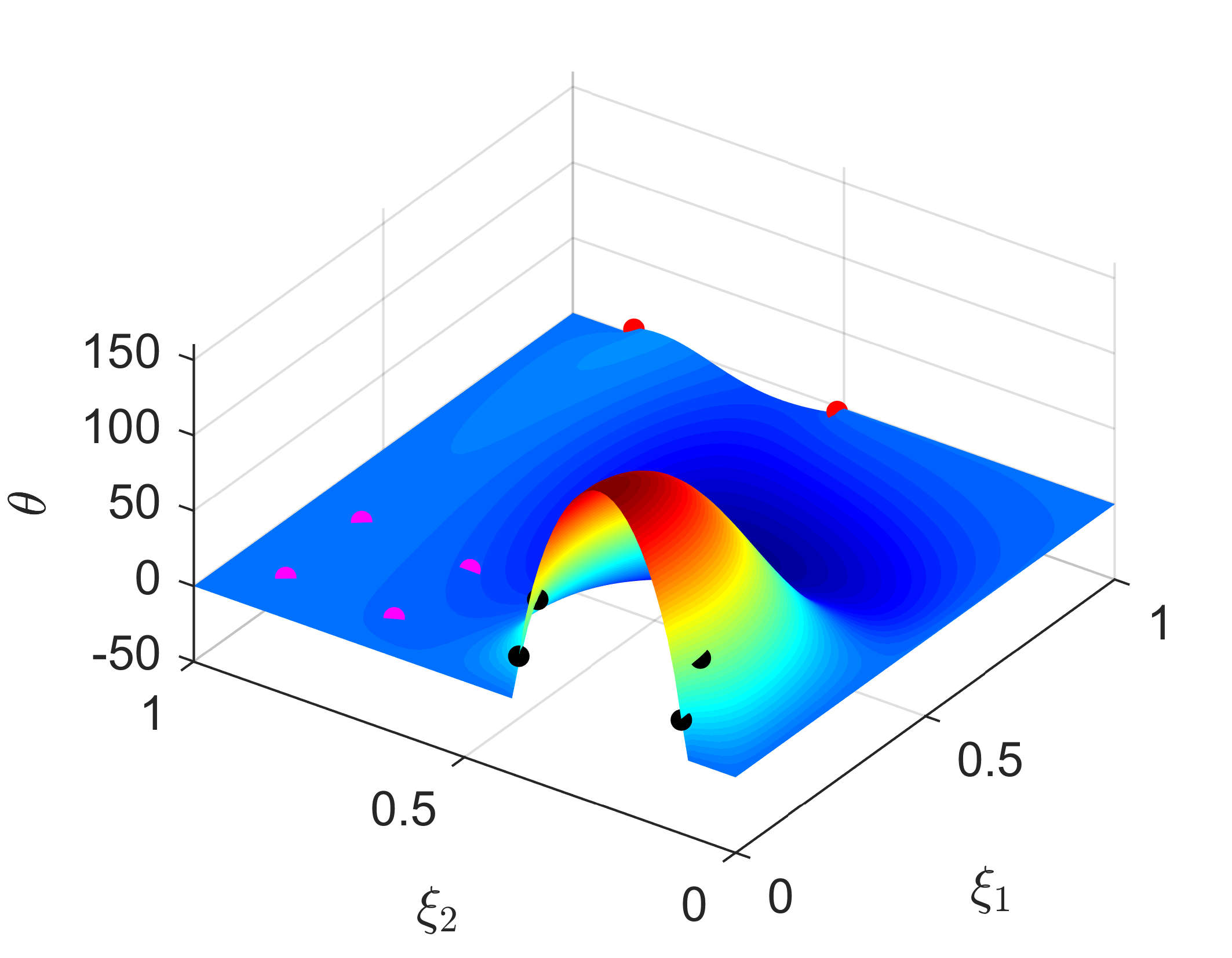}
	\end{subfigure} \hfill
	\caption{State of the plant at time $t=20$ with the observation setup \eqref{setup1} on the left and \eqref{setup2} on the right}
	\label{state}
\end{figure}

The temperature ranges in the room during the simulations vary greatly depending on the observation setup and the chosen controller. In the case of design \eqref{dcontroller}, model reduction does not affect the range significantly, but depending on the observation setup and the parameter choice, the temperature range is from $-28 \le \theta \le 23$ for the setup \eqref{setup1} with the parameter choice \eqref{choice2} to $-660 \le \theta \le 570$ with the setup \eqref{setup2} and the parameter choice \eqref{choice2}. Temperatures near the extrema are only observed during the first few seconds but the temperature fluctuations remain considerably larger with the observation setup \eqref{setup2} compared to \eqref{setup1} throughout the simulations. The temperature ranges for the low-gain design are $-20 \le \theta \le 19$ for the observation setup \eqref{setup1} and $-100 \le \theta \le 110$ for the observation setup \eqref{setup2}.

\section{CONCLUSION}\label{SCO}
We presented two different controllers for robust output regulation of a room temperature model. Based on numerical simulations, the dual observer-based controller outperforms the simpler low-gain robust controller design in the speed of convergence to the desired reference output. For the chosen initial state, its tracking performance also remained almost unchanged despite significant order reduction applied using balanced truncation. The dual observer-based controller also has the advantage over the low-gain one in that it can be designed for unstable systems as long as the system is exponentially stabilizable and detectable, whereas the low-gain controller requires exponential stability from the system.
In turn, the dual observer-based controller may cause comparably high temperature fluctuations within the system.

\bibliographystyle{plain}
\bibliography{References}

\end{document}